\documentclass[12pt]{amsart}
\usepackage{amscd}
\usepackage{mathrsfs}
\usepackage{a4wide}
\usepackage{color}
\usepackage{amssymb,amsmath,amsthm,amsfonts,latexsym}
\usepackage[utf8x]{inputenc}
\usepackage{bbm} 

\usepackage{amsmath,amsfonts,amssymb,amsthm} 
\usepackage{tikz}
\usepackage{hyperref}
\usepackage[all]{xy}

\newtheorem{theorem}{Theorem}[section]
\newtheorem{corollary}[theorem]{Corollary}

\newtheorem{claim}[theorem]{Claim}

\newtheorem{proposition}[theorem]{Proposition}
\newtheorem{definition}[theorem]{Definition}
\newtheorem{question}[theorem]{Question}

\def\leukfrac#1/#2{\leavevmode
               \kern.1em
                \raise.9ex\hbox{\the\scriptfont0 ${}_#1$}
                \hskip -1pt\kern-.1em
                /\kern-.15em\lower.10ex\hbox{\the\scriptfont0 ${}_#2$}}

\makeatletter
\def\diam{\mathop{\operator@font diam}\nolimits}
\makeatother

\title{Borel selector for hypergraphons}
\author{Jan Greb\'ik}

\begin{document}

\begin{abstract}
We show that there is a Borel way of choosing a representative of a $k$-uniform hypergraphon.
This extends the result of Orbanz and Szegedy \cite{orbanszegedy} where this was shown for graphons.
\end{abstract}

\maketitle

The limits of dense graph sequences, so called graphons, were introduced by Lovász and Szegedy in \cite{LS}.
A graphon $W$ is a symmetric measurable function from $[0,1]^2$ to $[0,1]$ i.e. an element of $L^\infty([0,1]^2,\lambda)$ where $\lambda$ is the Lebesgue measure.
Note that the space of all such functions $\mathcal{W}^2_0$ is $weak^*$-closed and bounded in $L^\infty([0,1]^2,\lambda)$ i.e. compact in the $weak^*$ topology, but there are several reasons why the $weak^*$ topology turns out not to be the right notion here.
Instead of that one considers the notion of cut distance $\delta_\Box$.
This is a pseudometric on the space of all graphons $\mathcal{W}^2_0$ and after making the 
quotient we obtain a compact metric space $(\tilde{\mathcal{W}}^2_0,\delta_\Box)$ (see \cite{LS}).
The selection problem of Sourav Chatterjee that is mentioned in \cite{NPS} asks if there is 
a map $S:(\tilde{\mathcal{W}}^2_0,\delta_\Box)\to \mathcal{W}^2_0$ such that $S(\tilde{W})\in \tilde{W}$ and $S$ is measurable or even continuous for some topology defined on $\mathcal{W}^2_0$ i.e. the $||.||_\infty$ norm, $||.||_1$ norm, the $weak^*$ topology or the cut norm.
It was shown in \cite{orbanszegedy} that there is such a Borel selector (or lifting) for graphons.
In this note we show that Borel selection exists also in the case of $k$-uniform hypergraphons, an analogue of graphons for $k$-uniform hypergraphs.
One of the differences between graphons and hypergraphons is the lack of a good notion of a cut distance.
Our argument is a bit longer than the argument from \cite{orbanszegedy}, where a Theorem about liftings of set-valued maps is used as a black box, but we construct the Borel selector explicitly using only elementary arguments.
We note that after having our result one can easily verify the assumptions on the set valued map as in \cite{orbanszegedy} and get the existence of a sequence of selectors which satisfy the same properties as in \cite{orbanszegedy}.
The question about continuous such selector remains open.

\section{Introduction}

Let $[n]=\{1,2,...,n\}$.
Let $A$ be a finite set then denote as $r(A,k)$ the set of all nonempty subsets of $A$ of size at most $k$.
We write $r(A)$ for $r(A,|A|)$ and $r_<(A)$ for $r(A,|A|-1)$.
A finite \emph{k-uniform hypergraph} $G=(V(G),E(G))$ on a set $V(G)$ is some subset $E(G)\subseteq V^{[k]}$ that is symmetric i.e. $(x_1,...x_k)\in E(V)$ if and only if $(x_{\sigma(1)},...,x_{\sigma(n)})\in E(V)$ for every $\sigma\in S_k$.
Let $F,H$ be two finite $k$-uniform hypergraphs then we define 
$$t(F,H)=\frac{\operatorname{hom}(F,H)}{|V(H)|^{|V(F)|}}.$$
We say that a sequence $\{G_n\}_{n<\omega}$ of finite $k$-uniform hypergraphs is convergent if $t(F,G_n)$ is convergent for every finite $k$-uniform hypergraph $F$.

The limit objects for such a converging sequences are certain symmetric measurable subsets of $[0,1]^{r([k])}$ that are called $k$-uniform hypergraphons.
Here symmetric means that the set is invariant under change of coordinates where the bijection between coordinates is induced by some $\sigma\in S_k$ i.e. $\sigma$ induces a permutation $\tilde{\sigma}$ of $r([k])$. 
Let $W$ be a $k$-uniform hypergraphon and $F$ a finite $k$-uniform hypergraph then we define the homomorphism density as
$$t(F,W)=\int_{[0,1]^{r(V(F),k)}}\prod_{A\in E(F)}W(x_{r(A)})dx.$$

\begin{definition}
Let $\mathcal{W}^k_0$ be the set of all $k$-uniform hypergraphons.
\end{definition}

We consider the metric $d_1$ induced by the $||.||_1$ norm on $\mathcal{W}^k_0$ i.e. $d_1(U,W)=||U-W||_1$.
Note that under this norm is $\mathcal{W}^k_0$ complete separable metric space, i.e. Polish.

\begin{claim}
The metric space $(\mathcal{W}^k_0,d_1)$ is complete.
\end{claim}
\begin{proof}
Assume that we have a Cauchy sequence $\{W_n\}_{n<\omega}\subseteq \mathcal{W}^k_0$.
Since $W_n\in L_1([0,1]^{r([k])},\lambda)$ for every $n< \omega$ there must be a limit of this sequence $W\in L_1$.
Since there is a subsequence of $\{W_n\}_{n<\omega}$ that converges pointwise almost everywhere to $W$ we see that $W$ is also in $\mathcal{W}^k_0$.
\end{proof}

We describe a group $G$ that acts on $\mathcal{W}^k_0$ continuously, the definition is from \cite{elekszegedy}.
First recall that the group $H$ of all measure preserving bijections of $([0,1]^{r([k])},\lambda)$ when equipped with the weak topology is a Polish group.
The weak topology is described as follows, a basic open neighborhood of $g\in H$ is given by a measurable $A\subseteq [0,1]^{r([k])}$ and $\epsilon>0$ and is defined as $\{h\in H:\lambda(gA\triangle hA)<\epsilon\}$.
Let $S\in r([k])$, we define a projection $\pi_S:[0,1]^{r([k])}\to[0,1]^{r(S)}$.
Let $\mathcal{A}_S$ be the $\sigma-$algebra on $[0,1]^{r([k])}$ that is the pull-back of the $\sigma-$algebra on $[0,1]^{r(S)}$ via $\pi_S$.
Define also $\mathcal{A}^*_S$ to be the $\sigma-$algebra generated by 
$$\bigcup_{T\subseteq S}\mathcal{A}_T.$$

\begin{definition}
We say that $\psi:[0,1]^{r([k])}\to [0,1]^{r([k])}$ is a structure preserving map if
\begin{itemize}
\item $\psi$ is measure preserving,
\item for every $\sigma\in S_k$ we have that $\tilde{\sigma} \psi=\psi\tilde{\sigma}$,
\item for every $S\in r([k])$ is $\psi^{-1}(\mathcal{A}_S)\subseteq \mathcal{A}_S$,
\item for every $S\in r([k])$, every measurable $A\subseteq [0,1]$ and $A^*\subseteq [0,1]^{r([k])}$, where $x\in A^*$ if and only if $x(S)\in A$, is $\psi^{-1}(A^*)$ independent of $\mathcal{A}^*_S$.
\end{itemize}
\end{definition}

We define $G$ to be the collection of all $g\in H$ that are structure preserving.

\begin{proposition}
The group $G$ is a closed subgroup of $H$ and the natural action of $G$ on $(\mathcal{W}^k_0,d_1)$, given by $x\in gW$ if and only if $g^{-1}(x)\in W$, is continuous.
\end{proposition}
\begin{proof}
Assume that $\{g_n\}_{n<\omega}\in G$ and $g_n\to g\in H$.
We show that all the conditions in the definition of structure preserving map are closed in the weak topology.

The first one is clear.
Assume that the second does not hold then the measurable set $A=\{x:\tilde{\sigma}g(x)\not=g\tilde{\sigma}(x)\}$ has positive measure.
Therefore there is a set $A_0\subseteq A$ with a positive measure such that $\tilde{\sigma}g(A_0)\cap g\tilde{\sigma}(A_0)=\emptyset$.
But also $g_n\tilde{\sigma}(A_0)=\tilde{\sigma}g_n(A_0)\to \tilde{\sigma}g_n(A_0)$ which clearly leads to a contradiction.

Note that the function $d(A,B)=\lambda(A\triangle B)$ is a pseudometric on the $\sigma-$algebra of measurable subsets of $[0,1]^{r([k])}$.
The third condition can be evaluated for every $A\in \mathcal{A}_S$ individually and follows easily since $\mathcal{A}_S$ is closed in the metric $d$.

The fourth condition can also be checked individually for every pair of the form $A^*$ and $B\in \mathcal{A}^*_S$ and follows easily from the definitions.

It remains to prove that the action is continuous.
We use here the standard fact from the descriptive set theory (see \cite{Gao}) that a action of a Polish group on a Polish space is continuous if and only if it is separately continuous.
Let $W\in\mathcal{W}^k_0$ and $g\in G$ then $gW\in \mathcal{W}^k_0$ which is guaranteed by the second condition in the definition of structure preserving map.
Let $g_n\to g$ and $W$ are given.
Then since $W$ is a measurable subset of $[0,1]^{r([k])}$ we have from the definition of the weak topology on $G$ that $g_nW\to gW$.
Assume on the other hand that $W_n\to W$ in $d_1$ and $g\in G$ are given.
Then $gW_n\to gW$ because in fact the map $g.\_:\mathcal{W}^k_0\to \mathcal{W}^k_0$ is even an isometry for $d_1$.
\end{proof}

We follow \cite{elekszegedy} and define several pseudometrics on $\mathcal{W}^k_0$.
Recall that the pull-back $W^\psi$ of $W\in \mathcal{W}^k_0$ under structural preserving map $\psi$ is defined as $W^{\psi}=\psi^{-1}(W)$.

\begin{definition}
Let $U,W\in \mathcal{W}^k_0$ and define
\begin{itemize}
\item $\delta_w(U,W)$ is the minimal $\alpha\in \mathbb{R}^+_0$ such that $|t(F,U)-t(F,W)|\le |E(F)|\alpha$ for every finite $k$-uniform hypergraph $F$,
\item $\delta_1(U,W)=\inf_{\psi,\varphi}\{d_1(U^{\varphi},W^{\psi})\}$ where $\psi,\varphi$ are strucutural preserving maps,
\item $\delta(U,W)=\sum_{F_n}\frac{|t(F_n,U)-t(F_n,W)|}{2^n}$ where $\{F_n\}_{n<\omega}$ is some fixed enumeration of all finite $k$-uniform hypergraphf.
\end{itemize}
\end{definition}

The basic properties of this pseudometrics from Section 4 of \cite{elekszegedy} are summarized in the next theorem.

\begin{theorem}[\cite{elekszegedy}]\label{section4}
Let $U,W\in \mathcal{W}^k_0$ then
\begin{itemize}
\item $\delta_w,\delta$ are pseudometrics,
\item $|t(F,U)-t(F,W)|\le |E(F)|d_1(U,W)$,
\item for any structure preserving map $\psi$ is $\delta_w(U,U^\psi)=0$,
\item $\delta_w(U,W)=0$ if and only if there are two structure preserving maps $\psi,\varphi$ such that $\lambda(U^{\psi}\triangle W^{\varphi})=0$,
\item $\delta_w(U,W)=0$ if and only if $\delta_1(U,W)=0$,
\item for every $\epsilon>0$ there is a $g\in G$ such that $d_1(U,gW)\le \delta_1(U,W)+\epsilon$.
\end{itemize}
\end{theorem}

We prove next some other properties that are not implicitly stated in \cite{elekszegedy} or the proof is omitted there but they follow easily using Theorem\ref{section4}.
Fix some enumeration of some countable dense subset of $G$ i.e. some $\{g_m\}_{m<\omega}\subseteq G$.

\begin{proposition}
Let $U,W\in \mathcal{W}^k_0$. Then
\begin{itemize}
\item $\delta(U,W)=0$ if and only if $\delta_1(U,W)=0$,
\item for every $\epsilon>0$ there is $m<\omega$ such that $d_1(U,g_mW)\le \delta_1(U,W)+\epsilon$,
\item $\delta_1$ is a Borel pseudoemtric i.e. it is a Borel function $(\mathcal{W}^k_0,d_1)\times(\mathcal{W}^k_0,d_1)\to \mathbb{R}^+_0$.
\end{itemize}
\end{proposition}
\begin{proof}
The first part follows from Theorem\ref{section4} since if $\delta_1(U,W)=0$ then $\delta_w(U,W)=0$ and from the definition of $\delta_w$ we have that $|t(F,U)-t(F,W)|\le |E(F)|0$ i.e. $\delta(U,W)=0$.
On the other hand if $\delta(U,W)=0$ then again from the definition we have that $\delta_w(U,W)=0$ and Theorem\ref{section4} gives us $\delta_1(U,W)=0$.

For the second part use Theorem\ref{section4} to produce such $g\in G$ and then use the continuity of the action of $G$ on $\mathcal{W}^k_0$ to get the desired $m<\omega$.

It remains to prove the third part.
To see that $\delta_1$ is a pseudometric take $U,W,V\in \mathcal{W}^k_0$.
Let $\epsilon>0$ be given and find structural preserving maps $\psi_0,\varphi_0,\psi_1,\varphi_1$ such that $d_1(U^{\psi_0},W^{\varphi_0})\le \delta_1(U,W)+\epsilon$ and $d_1(V^{\psi_1},W^{\varphi_1})\le \delta_1(V,W)+\epsilon$.
We have that $\delta_1(W^{\varphi_0},W^{\varphi_1})=0$ and therefore there is structural preserving map $\alpha,  \beta$ such that $\lambda (W^{\varphi_0\alpha}\triangle W^{\varphi_1\beta})=0$.
It is easy to see that now $d_1(U^{\psi_0\alpha},V^{\psi_1\beta})\le \delta_1(U,W)+\delta_1(W,V)+2\epsilon$.

To see that this pseudometric is Borel note that for a fixed $r\in \mathbb{R}^+_0$ and $\epsilon>0$ is the set 
$$\{(U,W)\in \mathcal{W}^k_0\times \mathcal{W}^k_0:\exists m<\omega \ d_1(U,g_mW)\in [r,r+\epsilon)  \ \& \  \forall k<\omega \  d_1(U,g_kW)\not<r\}$$
Borel because of the continuity of the action.
\end{proof}

Note that we have actually proved that the quotients of all of these pseudometrics have the same underlying set and we denote it as $\tilde{\mathcal{W}}^k_0=\mathcal{W}^k_0/\delta=\mathcal{W}^k_0/\delta_1$.
We also write $\tilde{W}$ to be the equivalence class of $W\in \mathcal{W}^k_0$ in $\tilde{\mathcal{W}}^k_0$.
Note however that the metrics are different.
It is a fundamental result see\cite{elekszegedy} that $(\tilde{\mathcal{W}}^k_0,\delta)$ is compact Polish space.

\begin{proposition}
The metric $\delta_1$ is complete on $\tilde{\mathcal{W}}^k_0$ and $\operatorname{Borel}(\tilde{\mathcal{W}}^k_0,\delta)\subseteq \operatorname{Borel}(\tilde{\mathcal{W}}^k_0,\delta_1)$.
\end{proposition}
\begin{proof}
Assume that $\{\tilde{W}_n\}_{n<\omega}$ is a Cauchy sequence with respect to $\delta_1$.
Choose some representatives $\{W_n\}_{n<\omega}$ which forms a cauchy sequence for $d_1$.
We know that $\mathcal{W}^k_0$ is complete and therefore there is a limit $W\in \mathcal{W}^k_0$ of this sequence.
Note that then $\tilde{W}$ is a limit of $\{\tilde{W}_n\}_{n<\omega}$.

Let $C\subseteq \tilde{\mathcal{W}}^k_0$ be closed in $\delta$.
We show that it is also closed in $\delta_1$.
Take any sequence $\tilde{W}_n\to_{\delta_1} \tilde{W}$ where $\{\tilde{W}_n\}_{n<\omega}\subseteq C$.
Take a representatives such that $W_n\to_{d_1} W$ and observe that from Theorem\ref{section4} it follows that 
$$\left|t(F,W_n)-t(F,W) \right|\le |E(F)|d_1(W_n,W)\to 0$$
for every $k$-uniform hypergraph $F$.
We have $\tilde{W}\in C$ and that finishes the proof.
\end{proof}

\section{Selector}

We show in this section that there is a Borel map $S:(\tilde{\mathcal{W}}^k_0,\delta)\to(\mathcal{W}^k_0,d_1)$ such that $S(\tilde{W})\in \tilde{W}$.

\begin{proposition}\label{borelmap}
The quotient map $F:(\mathcal{W}^k_0,d_1)\to (\tilde{\mathcal{W}}^k_0,\delta_1)$, where $F(W)=\tilde{W}$, is continuous.
\end{proposition}
\begin{proof}
Let $W\in \mathcal{W}^k_0$ and $\epsilon>0$ be given.
Then 
$$F^{-1}(\{\tilde{U}:\delta_1(\tilde{U},\tilde{W})<\epsilon\})=\{U:\exists m<\omega \       d_1(W,g_mU)<\epsilon\}$$
is clearly open.
\end{proof}

\begin{corollary}\label{borelmap2}
The map $F$ is continuous for the $\delta$ metric.
\end{corollary}

From now on when we write $\tilde{\mathcal{W}}^k_0$ we mean $(\tilde{\mathcal{W}}^k_0,\delta_1)$.
Consider now the map $F\times id:\mathcal{W}^k_0\times \tilde{\mathcal{W}}^k_0 \to \tilde{\mathcal{W}}^k_0\times \tilde{\mathcal{W}}^k_0$ defined as $F\times id(W,\tilde{U})=(\tilde{W},\tilde{U})$.
It follows from Proposition\ref{borelmap} that this map is Borel.
Therefore the sets
$$A_n=\{(W,\tilde{U})\in \mathcal{W}^k_0\times \tilde{\mathcal{W}}^k_0:\delta_1(\tilde{W},\tilde{U})<\frac{1}{2^n}\}$$
are Borel.
We have $graph(F)=\bigcap_{n<\omega}A_n$.
Pick some countable dense subset $\{\tilde{V}_l\}_{l<\omega}\subseteq \tilde{\mathcal{W}}^k_0$ and choose also $\{V_l\}_{l<\omega}\subseteq \mathcal{W}^k_0$ its concrete representation.
Define a sequence of maps $f_n:\tilde{\mathcal{W}}^k_0\to \mathcal{W}^k_0$ where
$f_n(\tilde{W})=V_l$ where $l$ is the minimal index such that $(V_l,\tilde{W})\in A_n$.
These maps are clearly Borel.

\begin{theorem}\label{selector}
There is a Borel map $S:(\tilde{\mathcal{W}}^k_0,\delta_1)\to (\mathcal{W}^k_0,d_1)$, where $S(\tilde{W})\in \tilde{W}$.
\end{theorem}
\begin{proof}
Let $h_n:\tilde{\mathcal{W}}^k_0\to \{g_m\}_{m<\omega}$ be defined as $h_n(\tilde{W})=g_p$ where $p$ is the smallest index such that $d_1(f_n(\tilde{W}),g_pf_{n+1}(\tilde{W}))<\frac{1}{2^n}$.
Define a sequence of maps $S_n:\tilde{\mathcal{W}}^k_0\to \mathcal{W}^k_0$ where
$$S_n(\tilde{W})=\prod_{i<n}h_i(\tilde{W})f_n(\tilde{W}).$$
Since for a fixed $g\in G$ is the map $g.\_:\mathcal{W}^k_0\to\mathcal{W}^k_0$ an isometry with respect to $d_1$ we have
$$d_1(S_n(\tilde{W}),S_{n+1}(\tilde{W}))=d_1(f_n(\tilde{W}),h_n(\tilde{W})f_{n+1}(\tilde{W}))<\frac{1}{2^n}$$
and therefore $\{S_n(\tilde{W})\}_{n<\omega}$ is a Cauchy sequence in $d_1$ for each $\tilde{W}$.
Moreover $S_n$ is Borel for each $n<\omega$ and $\delta_1(\lim_{n\to \infty}S_n(\tilde{W}),W)=0$ for each $W\in \mathcal{W}^k_0$.
Finally we can put
$$S(\tilde{W})=\lim_{n\to\infty}S_n(\tilde{W})$$
which is Borel because it is a pointwise limit of a Borel functions and has the desired properties.
\end{proof}

\begin{corollary}\label{transversal}
There is a Borel transversal $T\subseteq \mathcal{W}^k_0$ for the equivalence given by the zero distance in the pseudometric $\delta$.
\end{corollary}
\begin{proof}
Since the selector $S$ is a Borel function which is injective we have that $T=S(\tilde{\mathcal{W}}^k_0)$ is Borel in $\mathcal{W}^k_0$ by a classical Suslin Theorem.
From the definition of $S$ it follows that $T$ is a Borel transversal for the equivalence given by the zero distance in $\delta_1$.
By Theorem\ref{section4} the equivalence given by the zero distance is the same for $\delta$ and $\delta_1$.
\end{proof}

\begin{corollary}
There is a Borel map $S:(\tilde{\mathcal{W}}^k_0,\delta)\to (\mathcal{W}^k_0,d_1)$, where $S(\tilde{W})\in \tilde{W}$.
Moreover $\operatorname{Borel}(\tilde{\mathcal{W}}^k_0,\delta)=\operatorname{Borel}(\tilde{\mathcal{W}}^k_0,\delta_1)$.
\end{corollary}
\begin{proof}
It is enough to show the moreover part since then the map $S$ from Theorem\ref{selector} remains Borel for the $\delta$ metric.

By Proposition\ref{borelmap} and Corollary\ref{borelmap2} the $graph(F)$ is a Borel subset of $\mathcal{W}^k_0\times \tilde{\mathcal{W}}^k_0$ for $d_1\times\delta_1$ and $d_1\times\delta$ and so is the set $R=T\times \tilde{\mathcal{W}}^k_0$ where $T$ is the transversal from Corollary\ref{transversal}.
Therefore $R\cap graph(F)=graph^{-1}(S)$ is Borel for both $d_1\times \delta_1$ and $d_1\times \delta$.
Now take $A\subseteq \tilde{\mathcal{W}}^k_0$ Borel in $\delta_1$.
Then $B=F^{-1}(A)\cap T$ is Borel in $\mathcal{W}^k_0$.
Since $D=B\times \tilde{\mathcal{W}}^k_0\cap graph^{-1}(S)$ is Borel also for $d_1\times \delta$ and each horizontal section consists of at most one point we have that the projection of $D$ on the second coordinate is Borel for $\delta$ and is equal to $A$.
\end{proof}

We showed that even though the topology on $\tilde{\mathcal{W}}^k_0$ given by $\delta_1$ is stronger than the compact topology given by $\delta$ their Borel structures are the same.
A similar result holds for $\mathcal{W}^k_0$ and all usually used topologies i.e. $||.||_{\infty}$, $||.||_{1}$ or the $weak^*$ topology.
Since the Borel structures are the same the map $S$ remains Borel for each of them. 
It is however not clear whether there exist such a selector $S$ that is moreover continuous with respect to any mentioned topologies.
This remains open even for the case of graphons $\mathcal{W}^2_0$ where the compact topology is given by the cut distance $\delta_\Box$.

\begin{question}[Sourav Chatterjee \cite{NPS}]
Is there a continuous selector $S:(\tilde{\mathcal{W}}^2_0,\delta_\Box)\to (\mathcal{W}^2_0,weak^*)$?
\end{question}

Note also that sometimes it is more convenient to consider $k$-uniform hypergraphons not as measurable symmetric subsets of $[0,1]^{r([k])}$ but as a symmetric measurable functions $W:[0,1]^{r_<([k])}\to [0,1]$ with the $d_1$ metric (also from $||.||_1$ norm).
The natural quotient map $Q$ from the former set to the other is defined as 
$$Q(W)(x)=\int_{[0,1]} W(x,x_{[k]})dx_{[k]}$$
where $x$ is the vector index by $r_<([k])$ coordinates and $x_{[k]}$ is the last coordinate.
This map is continuous and so if we want our selector to choose rather from the latter set it is enough to compose it with $Q$.\\


\end{document}